\documentclass{amsart}
\usepackage{amsfonts}
\usepackage{graphicx}
\usepackage{amscd}
\usepackage{amsmath}
\usepackage{amssymb}

\makeatletter
\@namedef{subjclassname@2010}{%
  \textup{2010} Mathematics Subject Classification}
\makeatother

\theoremstyle{plain}
\newtheorem*{acknowledgement}{Acknowledgement}

\newtheorem{corollary}{\bf Corollary}
\newtheorem{definition}{\bf Definition}
\newtheorem{lemma}{\bf Lemma}

\newtheorem{proposition}{\bf Proposition}

\newtheorem{theorem}{\bf Theorem}
\newtheorem{conjecture}{\bf Conjecture}

\theoremstyle{definition}

\numberwithin{equation}{section}

\title[Critical Spaces]{On Static Manifolds and Related Critical Spaces with cyclic parallel Ricci tensor}

\author{H. Baltazar}
\author{A. Da Silva}

\address[H. Baltazar]{Departamento de Matem\'{a}tica, Universidade Federal do Piau\'{\i}\\
64049-550 Te\-re\-si\-na, Piau\'{\i}, Brazil.}
\email{halyson@ufpi.edu.br}

\address[A. da Silva]{Faculdade de Matem\'{a}tica, Universidade Federal do Par\'{a}\\
	66075-110 Be\-l\'{e}m, Par\'{a}, Brazil.}
\email{adamsilva@ufpa.br}


\begin{document}

\newcommand{\spacing}[1]{\renewcommand{\baselinestretch}{#1}\large\normalsize}
\spacing{1.2}

\begin{abstract}
The aim of this paper is to classify three dimensional compact Riemannian manifolds $(M^{3},g)$ that admits a non-constant solution to the equation
$$-\Delta f g+Hess f-fRic=\mu Ric+\lambda g,$$
for some special constants $(\mu, \lambda)$, under assumption that the manifold has cyclic parallel Ricci tensor. Namely, the structures that we will study here will be: positive static triples, critical metrics of the volume functional, and critical metrics of the total scalar curvature functional. We shall also classify $n$-dimensional critical metrics of the volume functional with non-positive scalar curvature and satisfying the cyclic parallel Ricci tensor condition.
\end{abstract}

\maketitle

\section{Introduction}
\label{intro}

The theory of critical metrics is a classical subject in differential geometry. The main reason is the direct relation with physical problems, for instance, it is well know that, considering a suitable functional, the equations of classical mechanics can be obtained via variational methods (see \cite{AbrMar,ARN}). Moreover, as another application, Hilbert in \cite{hilbert} gave a different perspective to understand the equations of general relativity, more precisely, working with total scalar curvature functional his was able to show that such equations can be recovered from the action of that functional. In particular, critical metrics of the Einstein-Hilbert functional, as it is also known, are in fact Einstein manifolds. This result becomes more refined if we consider a more restrictive space in order to obtain the same response, namely, the space of all metrics with constant scalar curvature and unit volume. It has been conjectured that the solutions of such critical point equation are Einstein metrics. Investigating this conjecture will be one of the main goal of our work.

Before to proceed it is fundamental to remember the following 3-tensor in terms of the Ricci tensor
\begin{equation}\label{auxD}
D(X,Y,Z)=\nabla_{X}Ric(Y,Z)+\nabla_{Y}Ric(Z,X)+\nabla_{Z}Ric(X,Y),
\end{equation}
for any $X,Y,Z\in\mathfrak{X}(M)$. For sake of simplicity, we now rewrite the above equation in the tensorial language as follows
$$D_{ijk}=\nabla_{i}R_{jk}+\nabla_{j}R_{ki}+\nabla_{k}R_{ij}.$$
We say that a Riemannian manifold has cyclic parallel Ricci tensor if $D_{ijk}=0.$ Such concept was introduced by A. Gray in \cite{Gray} and clearly generalizes the definition of an Einstein manifold. Moreover, it is obvious that if the Ricci tensor of a Riemannian manifold is parallel, then it satisfies the above condition. But, the converse statement is not true, see \cite{Gray} and \cite{Jelonek} for more details.

The purpose of this paper is to investigate some special structures under cyclic parallel Ricci tensor condition. In order to unify such structures, we shall consider an $n$-dimensional compact Riemannian manifold $(M,g)$ which admits a smooth non-constant solution $f$ to the equation
\begin{equation}\label{eq:criticalspace}
-\Delta f g+Hess f-fRic=\mu Ric+\lambda g,
\end{equation}
where $Ric,\Delta$ and $Hess$ stand, respectively, for the Ricci tensor, the Laplacian operator and the Hessian form on $M$ and $(\mu,\lambda)$ are two real constants given by $(0,0)$ or $(0,1)$ or $(1,-R/n).$ Moreover, if $(\mu,\lambda)=(0,0)$ or $(0,1),$ the manifold considered here must be compact with non-empty boundary $\partial M.$ Such structures are well-known as positive static triple or Miao-Tam critical metric, respectively. Otherwise, if $(\mu,\lambda)=(1,-R/n),$  we will be dealing with the well-known CPE metric and in this case we take a compact manifold $M$ without boundary. (see Definition~\ref{CriticalSpace:DEF} in Section~\ref{Preliminaries} for a complete description).

In what follows, we should describe a brief introduction of each case separately. We start with a metric that emerges as critical point of the total scalar curvature functional, $g\mapsto\int_{M}R_{g}dM_{g}$ where $R$ and $dM$ stand for the scalar curvature and the volume form of the metric $g$, respectively, restricted to the space of of smooth metrics $g$ on a given compact manifold $M^{n},$ $n\geq3$, with constant scalar curvature and unit volume, for simplicity CPE metrics.

As a mentioned before, it was conjectured in 1980's that every CPE metric must be Einstein. Moreover, considering a non-trivial solution of Eq. (\ref{eq:criticalspace}) then, we can apply the well-known Obata's theorem (cf. \cite{Obata}), in order to deduce that $(M^{n},g)$ should be isometric to a round sphere. The conjecture was established in \cite{besse} and we would like to present here in the following way.

\begin{conjecture}\label{conjCPE}
A CPE metric is always Einstein.
\end{conjecture}

In \cite{CHY12} Chang, Hwang and Yun, proved the conjecture for $n$-dimensional manifolds with parallel Ricci tensor. Recently the first author obtained the same result with a different approach, see \cite[Remark 2]{B17} for more details.

\begin{theorem}[Chag-Hwang-Yun, \cite{CHY12}]\label{CHYCPE}
Let $(M, g)$ be an $n$-dimensional compact Riemannian manifold, $n\geq3$, and $f$ be a non-trivial solution of CPE. If its Ricci tensor is parallel,
then $(M, g)$ is isometric to a standard sphere.
\end{theorem}
Let us highlight that every manifold with parallel Ricci tensor has harmonic curvature (i.e., $\nabla_{i}R_{ijkl}=0$). However, there are examples proving that the converse statement is not true see, for instance, \cite{derd} and \cite{derd2}. In this sense, Yun et al. \cite{CHY14,CHY16} was able to obtain the conjecture under the harmonicity of Riemann curvature tensor (see Eq. (\ref{Bianchi})). In addition, since a CPE metric has constant scalar curvature, it is not difficult to check that, the CPE conjecture is also true under the weaker condition of harmonic Weyl tensor.

Motivated by the historical development on the study of CPE metrics, we shall obtain an affirmative answer to the CPE conjecture under cyclic parallel Ricci tensor condition in dimension three. We now state our first result.

\begin{theorem}\label{thmCPE}
The Conjecture~\ref{conjCPE} is true for $3$-dimensional manifolds with cyclic parallel Ricci tensor.
\end{theorem}

Next, we must report some facts about static metrics. The study of static spacetime has been of substantial interest from both the physical and mathematical aspects. From a physical point of view, static potentials appears in General Relativity as global solutions to Einstein equations and therefore that interest is motivated by their role in the study of different problems in such area see, for example, \cite{FM75} and \cite{FischerMarsden}. From a mathematical point of view, that interest is also motivated by the fact that these structures exhibit nice geometric properties. Actually, Kobayashi \cite{kobayashi} and Lafontaine \cite{lafontaine} working independently, proved a classification result for locally conformally flat static triples in all dimensions. For our purpose, since $3$D-static triples satisfying the parallel Ricci tensor assumption implies in locally conformally flat metrics, we state their result in dimension $n=3$ under such condition. More precise, they established the following result.

\begin{theorem}[Kobayashi \cite{kobayashi}, Lafontaine \cite{lafontaine}]\label{KL}
Let $(M^{3},\,g,\,f)$ be an $3$-dimensional positive static triple with scalar curvature $R=6.$ Suppose that $(M^{3},\,g)$ has parallel Ricci tensor, then $(M^{3},\,g,\,f)$ is covered by a static triple equivalent to one of the following two static triples:

\begin{enumerate}
\item The standard hemisphere with canonical metric $(\mathbb{S}^{3}_{+},g_{\mathbb{S}^{3}}).$
\item The standard cylinder over $\mathbb{S}^{2}$ with the product metric
$$\Big(M=\Big[0,\frac{\pi}{\sqrt{3}}\Big]\times\mathbb{S}^{2},\; g=dt^{2}+\frac{1}{3}g_{\mathbb{S}^{2}}\Big).$$
\end{enumerate}

\end{theorem}

At this point, it is important to remember the {\it cosmic no-hair conjecture} which states that the standard round hemisphere is the only compact simply connected static triple with positive scalar curvature and connected boundary. Such conjecture was formulated by Boucher, Gibbons and Horowitz in \cite{BGH} and many partial answers were achieved see, for instance, \cite{Amb,balt18,adam,Montiel,kobayashi,lafontaine,QY}.

Here, inspired by this last result on the study of static triples, we shall replace the assumption that the manifold has parallel Ricci tensor in the Kobayashi-Lafontaine result by the cyclic parallel Ricci tensor condition, which is weaker that the former. So, our result can be seen as a partial answer to
{\it cosmic no-hair conjecture}. To be precise, we have established the following result.

\begin{theorem}\label{Static}
Let $(M^{3},\,g,\,f)$ be an $3$-dimensional positive static triple with scalar curvature $R=6.$ Suppose that $(M^{3},\,g)$ has cyclic parallel Ricci tensor, then $(M^{3},\,g,\,f)$ is covered by a static triple equivalent to one of the following two static triples:

\begin{enumerate}
\item The standard hemisphere with canonical metric $(\mathbb{S}^{3}_{+},g_{\mathbb{S}^{3}}).$
\item The standard cylinder over $\mathbb{S}^{2}$ with the product metric
$$\Big(M=\Big[0,\frac{\pi}{\sqrt{3}}\Big]\times\mathbb{S}^{2},\; g=dt^{2}+\frac{1}{3}g_{\mathbb{S}^{2}}\Big).$$
\end{enumerate}

\end{theorem}

Finally, we would like to describe some facts about Miao-Tam critical metrics. Such structures have investigated by Miao and Tam in \cite{miaotam}, where the authors provided sufficient condition for a metric to be a critical point of the volume functional on the space of constant scalar curvature metrics with a prescribed boundary metric. For more details, we refer the reader to \cite[Theorem 5]{miaotam}.

In 2011, Miao and Tam posed the question of whether there exist non-constant sectional curvature Miao-Tam critical metrics on a compact manifold whose boundary is isometric to a standard round sphere, see \cite{miaotamTAMS}. In order to give a partial answer to this problem, they showed that an Einstein Miao-Tam critical metric $(M^{n},g,f)$ must be isometric to a geodesic ball in a simply connected space form $\mathbb{R}^{n}$, $\mathbb{H}^{n}$ or $\mathbb{S}^{n}$. In this special case the restriction on the boundary was not necessary.

At the same time, the first author and Ribeiro Jr. in \cite{balt18} studied these critical metrics under parallel Ricci condition. More precisely, they obtained the following result.
\begin{theorem}[Baltazar-Ribeiro, \cite{balt18}]
\label{thmMT}
Let $(M^{n},g,f)$ be a compact, oriented, connected Miao-Tam critical metric with parallel Ricci tensor and smooth boundary $\partial M.$ Then $(M^{n},g)$ is isometric to a geodesic ball in a simply connected space form $\mathbb{R}^{n}$, $\mathbb{H}^{n}$ or $\mathbb{S}^{n}.$
\end{theorem}

In the sequel, we shall prove the rigidity for a $3$-dimensional Miao-Tam critical metric with cyclic parallel Ricci tensor. More precisely, we have the following result.

\begin{theorem}\label{thmMT3D}
Let $(M^{3},g,f)$ be a Miao-Tam critical metric with cyclic parallel Ricci tensor. Then $(M^3,\,g)$ is isometric to a geodesic ball in a simply connected space form $\Bbb{R}^{3},$ $\Bbb{H}^{3}$ or $\Bbb{S}^{3}.$
\end{theorem}

Based on the previous result, it is natural to ask what occurs in higher dimension. To do so, inspired by the ideas developed in \cite{WW18}, see also \cite{baltazar,balt18}, we shall prove a rigidity result for a $n$-dimensional Miao-Tam critical metric with non-positive scalar curvature and satisfying the cyclic parallel Ricci tensor condition.

\begin{theorem}\label{thmMTnD}
Let $(M^n,\,g,\,f)$ be a Miao-Tam critical metric with non-positive scalar curvature and cyclic parallel Ricci tensor. Then $(M^n,\,g)$ is isometric to a geodesic ball in a simply connected space form $\Bbb{R}^{n}$ or $\Bbb{H}^{n}.$
\end{theorem}

In particular, this shows that the additional assumptions of Theorem 1.1, Theorem 1.2 and Theorem 1.3 in \cite{WW18} are not necessary and then, Theorem~\ref{thmMTnD} improves the results in \cite{WW18} for non-positive scalar curvature case.

\section{Preliminaries}
\label{Preliminaries}
Throughout this section, we recall four special tensors in the study of curvature for a Riemannian manifold $(M^n,\,g),\,n\ge 3.$  The first one is the Weyl tensor $W$ which is defined by the following decomposition formula
\begin{eqnarray}
\label{weyl}
R_{ijkl}&=&W_{ijkl}+\frac{1}{n-2}\big(R_{ik}g_{jl}+R_{jl}g_{ik}-R_{il}g_{jk}-R_{jk}g_{il}\big) \nonumber\\
 &&-\frac{R}{(n-1)(n-2)}\big(g_{jl}g_{ik}-g_{il}g_{jk}\big),
\end{eqnarray}
where $R_{ijkl}$ stands for the Riemann curvature operator $Rm,$ whereas the second one is the Cotton tensor $C$ given by
\begin{equation}
\label{cotton} \displaystyle{C_{ijk}=\nabla_{i}R_{jk}-\nabla_{j}R_{ik}-\frac{1}{2(n-1)}\big(\nabla_{i}R
g_{jk}-\nabla_{j}R g_{ik}).}
\end{equation} Note that $C_{ijk}$ is skew-symmetric in the first two indices and trace-free in any two indices. In addition, it is well known that for $n=3,$ the Weyl tensor vanishes identically, while $C_{ijk}=0$ if and only if $(M^{3},g)$ is locally conformally flat; for $n\geq4$ we have that $W_{ijkl}=0$ if and only if $(M,g)$ is locally conformally flat.

For our purpose we also remember that, as consequence of Bianchi identity, the divergence of the Riemann tensor is given by
\begin{equation}\label{Bianchi}
({\rm div} Rm)_{jkl}=\nabla_kR_{jl}-\nabla_lR_{jk}.
\end{equation}
It is not difficult to see that, these two tensors above are related as follows
\begin{equation}\label{cottondivRm}
\displaystyle{|{\rm div} Rm|^{2}=|C_{ijk}|^{2}+\frac{1}{2(n-1)}|\nabla R|^{2}.}
\end{equation}

We also recall that, from commutation formulas (Ricci identities), for any Riemannian manifold $M^n$ we have
\begin{equation}\label{idRicci}
\nabla_i\nabla_j R_{kl}-\nabla_j\nabla_i R_{kl}=R_{ijks}R_{sl}+R_{ijls}R_{ks},
\end{equation} for more details see \cite{chow}.

In the sequel, in order to unify our results, we will define a critical space as one of the three spaces described in the introduction as follows.
\begin{definition}\label{CriticalSpace:DEF}
A critical space is a 3-tuple $(M^n,\,g,\,f),$ where $(M^{n},\,g),$ $n\geq3,$ is a compact Riemannian manifold with constant scalar curvature which admits a smooth non-constant solution $f$ to the equation
\begin{equation}\label{eq:CS}
-(\Delta f)g+Hess f-fRic=\mu Ric+\lambda g,
\end{equation}
where $Ric,\Delta$ and $Hess$ stand, respectively, for the Ricci tensor, the Laplacian operator and the Hessian form on $M$ and $(\mu,\lambda)$ are two real constants given by $(0,0)$ or $(0,1)$ or $(1,-R/n).$ Moreover, if $(\mu,\lambda)=(0,0)$ or $(0,1),$ the manifold considered here must be compact with non-empty boundary $\partial M$ and its potential function $f$ must be nonnegative and $f^{-1}(0)=\partial M.$ Otherwise, if $(\mu,\lambda)=(1,-R/n),$ the manifold $M$ that we will treat here must be without boundary.
\end{definition}

In particular, tracing (\ref{eq:CS}) we have
\begin{equation}
\label{eqtrace} \Delta f +\frac{(\mu+f)R+\lambda n}{n-1}=0.
\end{equation}
For sake of simplicity, we now rewrite equation (\ref{eq:CS}) in the tensorial language as follows
\begin{equation}\label{eq:tensorialCS}
-(\Delta f)g_{ij}+\nabla_{i}\nabla_{j}f-fR_{ij}=\mu R_{ij}+\lambda g_{ij}.
\end{equation} Moreover, by using (\ref{eqtrace}) we may deduce
\begin{equation}
\label{IdRicHess} (\mu+f)\mathring{Ric}=\mathring{Hess f},
\end{equation} where $\mathring{T}$ stands for the traceless of $T.$

Next, using Equation (\ref{eq:tensorialCS}), it is not difficult to check that the following formula holds on $M.$
\begin{lemma}\label{L1}
Let $(M^n,\,g,\,f,\,\mu,\,\lambda),$ be a critical space. Then we have:
$$(\mu+f)C_{ijk}=R_{ijkl}\nabla_{l}f+\frac{R}{n-1}(\nabla_{i}fg_{jk}-\nabla_{j}fg_{ik})-(\nabla_{i}fR_{jk}-\nabla_{j}f R_{ik}).$$
\end{lemma}
The proof of the previous result is essentially the same as the proof of Lemma 1 in \cite{BDR}.

Now, we recall the following auxiliary $3$-tensor defined in \cite{BDR},
\begin{eqnarray}\label{TensorT}
T_{ijk}&=&\frac{n-1}{n-2}(R_{ik}\nabla_{j}f-R_{jk}\nabla_{i}f)-\frac{R}{n-2}(g_{ik}\nabla_{j}f-g_{jk}\nabla_{i}f)\nonumber\\
&&+\frac{1}{n-2}(g_{ik}R_{js}\nabla_{s}f-g_{jk}R_{is}\nabla_{s}f).
\end{eqnarray}
Note that $T_{ijk}$ has the same symmetry properties as the Cotton tensor and, with a straightforward computation, we verify that
\begin{equation}\label{CTW}
(\mu+f)C_{ijk}=T_{ijk}+W_{ijkl}\nabla_{l}f.
\end{equation}

\section{Key Lemmas}
This section is devoted to quote some key lemmas which will be essential to establish our classification results.

The first result was proved in \cite{BDRR} and provide a integral formula involving the squared norm of the traceless Ricci tensor. In what follows we shall consider $(M^{n},g,f)$, ($n\geq3$), as a $n$-dimensional critical space with smooth boundary $\partial M$ and constant scalar curvature $R.$ Thus, the potential function $f$ satisfies $f>0$ in the interior of $M,$  if $\nu$ denotes the outward unit normal to $\partial M,$ then $\langle\nabla f,\nu\rangle<0$ on each connected component of $\partial M,$ (see for instance, \cite{miaotam}). To be more precise, the outward normal vector field along the boundary must be given by $\nu=-\nabla f/|\nabla f|.$ Since the proof is very short, we include it here for the sake of completeness.

\begin{lemma}\label{AuxRicci}
Let $(M^n,\,g,\,f,\,\mu,\,\lambda),$ be a critical space. If $(\mu,\lambda)=(0,0)$ or $(0,1),$ then we have:
\begin{eqnarray*}
\int_{M} f|\mathring{Ric}|^{2}dM_{g}+\int_{\partial M} \frac{1}{|\nabla f|}\mathring{Ric}(\nabla f, \nabla f)dS=0.
\end{eqnarray*}
\end{lemma}
\begin{proof}
Since $M$ has constant scalar curvature, we can use the twice contracted second Bianchi identity jointly with Equation (\ref{eq:tensorialCS}) to arrive at
$$\nabla_{i}(\mathring{R}_{ij}\nabla_{j}f)=\mathring{R}_{ij}\nabla_{i}\nabla_{j}f=f|\mathring{Ric}|^{2}.$$

To conclude, we apply the divergence theorem to get
$$\int_{M} f|\mathring{Ric}|^{2}dM_{g}=\int_{\partial M}\mathring{Ric}(\nabla f,\nu)dS,$$
where $\nu$ denotes the outward unit normal to $\partial M.$ This is what we wanted to prove.
\end{proof}

In order to proceed, by direct computation using the definition of tensor $D_{ijk},$ see Eq. (\ref{auxD}), we can deduce that
\begin{eqnarray}\label{nablaRIC}
\frac{1}{3}|D_{ijk}|^{2}&=&D_{ijk}\nabla_{i}R_{jk}\nonumber\\
&=&|\nabla Ric|^{2}+\nabla_{i}R_{jk}\nabla_{j}R_{ki}+\nabla_{i}R_{jk}\nabla_{k}R_{ij}\nonumber\\
&=&|\nabla Ric|^{2}+2\nabla_{i}R_{jk}\nabla_{j}R_{ki},
\end{eqnarray}
where in the last step we just use the symmetry of the Ricci tensor. Furthermore, Eq. (\ref{cotton}) jointly with fact that $(M,g)$ has constant scalar curvature, allow us to conclude that
\begin{equation}\label{Caux}
|C_{ijk}|^{2}=|\nabla_{i}R_{jk}-\nabla_{j}R_{ik}|^{2}=2|\nabla Ric|^{2}-2\nabla_{i}R_{jk}\nabla_{j}R_{ik}.
\end{equation}
Thus, we sum the expressions obtained in (\ref{nablaRIC}) and (\ref{Caux}) to deduce a key property in order to obtain our results, namely
\begin{equation}\label{normaD}
\frac{1}{3}|D_{ijk}|^{2}+|C_{ijk}|^{2}=3|\nabla Ric|^{2}.
\end{equation}

With these considerations in mind, we may announce our second lemma of this section.

\begin{lemma}\label{L4}
Let $(M^n,\,g,\,f,\,\mu,\,\lambda),$ be a critical space. Then we have:
\begin{eqnarray*}
\label{kp}
\frac{1}{3}(\mu+f)|D_{ijk}|^{2}&=&2 (\mu+f)|\nabla Ric|^{2}+\frac{1}{2}{\rm div}((\mu+f)\nabla|Ric|^{2})\\
&&-\nabla_{i}((\mu+f)C_{ijk}R_{jk}+R_{ik}R_{kj}\nabla_{j}f+R_{ijkl}\nabla_{l}f R_{jk}).
\end{eqnarray*}
\end{lemma}
\begin{proof}
Denote
$$X_{i}=R_{ik}R_{kj}\nabla_{j}f+R_{ijkl}\nabla_{l}f R_{jk}.$$
Then, using (\ref{eq:tensorialCS}) and (\ref{Bianchi}) joint with fact that $M$ has constant scalar curvature, we immediately deduce
\begin{eqnarray*}
{\rm div}X&=& R_{ik}\nabla_{i}R_{kj}\nabla_{j}f+(R_{ik}R_{kj}\nabla_{i}\nabla_{j}f+R_{ijkl}R_{jk}\nabla_{i}\nabla_{l}f)\\
&&\nabla_{i}R_{ijkl}R_{jk}\nabla_{l}f+R_{ijkl}\nabla_{i}R_{jk}\nabla_{l}f\\
&=&2\nabla_{i}R_{jk}R_{ik}\nabla_{j}f+(\mu+f)(R_{ik}R_{kj}R_{ij}-R_{ijkl}R_{ik}R_{jl})\\
&&-\frac{1}{2}\langle \nabla|Ric|^{2},\nabla f\rangle+\frac{1}{2}R_{ijkl}(\nabla_{i}R_{jk}-\nabla_{j}R_{ik})\nabla_{l}f.
\end{eqnarray*}
Here, we change some indices just for simplicity. Hence, by (\ref{cotton}), we have
\begin{eqnarray}\label{EqdivX}
{\rm div}X&=&2C_{ijk}R_{ik}\nabla_{j}f+(\mu+f)(R_{ik}R_{kj}R_{ij}-R_{ijkl}R_{ik}R_{jl})\nonumber\\
&&+\frac{1}{2}\langle \nabla|Ric|^{2},\nabla f\rangle+\frac{1}{2}R_{ijkl}\nabla_{l}fC_{ijk}.
\end{eqnarray}
Next, using Lemma~\ref{L1}, we can rewrite (\ref{EqdivX}) as follows
\begin{eqnarray}\label{EqdivX2}
{\rm div}X-\frac{(\mu+f)}{2}|C_{ijk}|^{2}&=&C_{ijk}R_{ik}\nabla_{j}f+(\mu+f)(R_{ik}R_{kj}R_{ij}-R_{ijkl}R_{ik}R_{jl})\nonumber\\
&&+\frac{1}{2}\langle \nabla|Ric|^{2},\nabla f\rangle.
\end{eqnarray}
On the other hand, a straightforward computation considering the Equations (\ref{Caux}) and (\ref{idRicci}), allow us to conclude that
\begin{eqnarray*}
\frac{(\mu+f)}{2}|C_{ijk}|^{2}&=& (\mu+f)|\nabla Ric|^{2}-(\mu+f)\nabla_{i}R_{jk}\nabla_{j}R_{ik}\\
&=&(\mu+f)|\nabla Ric|^{2}-\nabla_{j}((\mu+f)\nabla_{i}R_{jk}R_{ik})\\
&&+C_{ijk}R_{ik}\nabla_{j}f+\frac{1}{2}\langle\nabla|Ric|^{2},\nabla f\rangle\\
&&+(\mu+f)(R_{ik}R_{kj}R_{ij}-R_{ijkl}R_{ik}R_{jl}),
\end{eqnarray*}
i.e.,
\begin{eqnarray*}
&&\frac{(\mu+f)}{2}|C_{ijk}|^{2}-(\mu+f)|\nabla Ric|^{2}+\nabla_{j}((\mu+f)\nabla_{i}R_{jk}R_{ik})=\\
&=&C_{ijk}R_{ik}\nabla_{j}f+\frac{1}{2}\langle\nabla|Ric|^{2},\nabla f\rangle+(\mu+f)(R_{ik}R_{kj}R_{ij}-R_{ijkl}R_{ik}R_{jl}).
\end{eqnarray*}
Therefore, comparing the expressions obtained in last identity and (\ref{EqdivX2}), we obtain
\begin{eqnarray*}
(\mu+f)|C_{ijk}|^{2}&=&(\mu+f)|\nabla Ric|^{2}-\nabla_{j}((\mu+f)\nabla_{i}R_{jk}R_{ik})\\
&&\nabla_{i}(R_{ik}R_{kj}\nabla_{j}f+R_{ijkl}\nabla_{l}f R_{jk}),
\end{eqnarray*}
which combined with (\ref{normaD}) gives the desired result.
\end{proof}

We conclude this section by recalling the Bochner-type formula obtained by the first author and Ribeiro Jr. in \cite{balt18}, that for our purpose, we write such formula considering an arbitrary critical space.

\begin{proposition}\label{bochnerAUX}
Let $(M^n,\,g,\,f,\,\mu,\,\lambda),$ be a critical space. Then we have:
\begin{eqnarray}
\frac{1}{2}{\rm div}((\mu+f)\nabla|Ric|^{2})&=&(\mu+f)\Big(\frac{n-2}{n-1}|C_{ijk}|^{2}+|\nabla Ric|^{2}\Big)+ \frac{n\lambda}{n-1}|\mathring{Ric}|^{2}\nonumber\\
&&+(\mu+f)\Big(\frac{2}{n-1}R|\mathring{Ric}|^{2}+\frac{2n}{n-2}tr(\mathring{Ric}^{3})\Big)\nonumber\\
&&-\frac{n-2}{n-1}W_{ijkl}\nabla_{l}fC_{ijk}-2(\mu+f)W_{ijkl}R_{ik}R_{jl},
\end{eqnarray} where $C$ stands for the Cotton tensor, $W$ is the Weyl tensor and $\mathring{Ric}$ is the traceless Ricci tensor.
\end{proposition}

\section{Critical Spaces with Cyclic Parallel Ricci Tensor}

\subsection{3D-Critical spaces}

In this section we shall prove the Theorems \ref{thmCPE}, \ref{Static} and \ref{thmMT3D}  announced in the introduction. To do so, we shall present a fundamental integral formula that, when restricted to the three dimensional case, will be able to classify the critical spaces under cyclic parallel Ricci tensor condition.

\begin{theorem}\label{thmMainC}
Let $(M^n,\,g,\,f,\,\mu,\,\lambda),$ $n\geq3$,  be a critical space. Then we have:
\begin{eqnarray*}
\int_{M} (\mu+f)|D_{ijk}|^{2}dM_{g}&=&\frac{18(2n-3)}{3n-5}\int_{M} (\mu+f)|\nabla Ric|^{2}dM_{g}+\frac{9n\lambda}{3n-5}\int_{M} |\mathring{Ric}|^{2}dM_{g}\\
&&+\frac{6(n-1)}{3n-5}\int_{M}\left((\mu+f)\langle {\rm div}D,Ric\rangle-\frac{n-2}{n-1}\langle \iota_{\nabla f}W,C\rangle\right) dM_{g},
\end{eqnarray*}
where $\iota$ stands for the interior product $(\iota_{\nabla f}W)_{ijk}=W_{ijkl}\nabla_{l}f.$
\end{theorem}
\begin{proof}
By direct computation using the definition of tensor $D$, see Equation (\ref{auxD}), we obtain
\begin{equation}\label{DRf}
D_{ijk}R_{jk}\nabla_{i}f=\frac{1}{2}\langle\nabla|Ric|^{2},\nabla f\rangle+2\nabla_{j}R_{ki}R_{jk}\nabla_{i}f.
\end{equation}
Consequently, applying the divergence theorem combined with (\ref{idRicci}), we have that
\begin{eqnarray}\label{auxintDR}
\int_{M}D_{ijk}R_{jk}\nabla_{i}f dM_{g}&=&\frac{1}{2}\int_{M}\langle\nabla|Ric|^{2},\nabla f\rangle dM_{g}-2\int_{M}\nabla_{j}R_{ki}\nabla_{i}R_{jk}(\mu+f)dM_{g}\nonumber\\
&&-2\int_{M}\nabla_{i}\nabla_{j}R_{ki}R_{jk}(\mu+f)dM_{g}\nonumber\\
&=&\frac{1}{2}\int_{M}\langle\nabla|Ric|^{2},\nabla f\rangle dM_{g}-2\int_{M}\nabla_{j}R_{ki}\nabla_{i}R_{jk}(\mu+f)dM_{g}\nonumber\\
&&-2\int_{M}(R_{ij}R_{ik}R_{jk}-R_{ijkl}R_{ik}R_{jl})(\mu+f)dM_{g}.
\end{eqnarray}
On the other hand, from definition of Cotton tensor and Equation (\ref{CTW}) we get another useful identity, namely
\begin{eqnarray*}
2\nabla_{j}R_{ik}R_{jk}\nabla_{i}f&=&2(C_{jik}+\nabla_{i}R_{jk})R_{jk}\nabla_{i}f\\
&=&-C_{ijk}(R_{jk}\nabla_{i}f-R_{ik}\nabla_{j}f)+\langle\nabla|Ric|^{2},\nabla f\rangle\\
&=&\frac{n-2}{n-1}C_{ijk}T_{ijk}+\langle\nabla|Ric|^{2},\nabla f\rangle\\
&=&\frac{n-2}{n-1}((\mu+f)|C_{ijk}|^{2}-\langle \iota_{\nabla f}W,C\rangle)+\langle\nabla|Ric|^{2},\nabla f\rangle
\end{eqnarray*}
which substituting in (\ref{DRf}) joint with (\ref{normaD}) becomes
\begin{eqnarray}\label{Ricnablaf}
\frac{1}{2}\langle\nabla|Ric|^{2},\nabla f\rangle&=&\frac{1}{3}D_{ijk}R_{jk}\nabla_{i}f-\frac{n-2}{(n-1)}(\mu+f)\left(|\nabla Ric|^{2}-\frac{1}{9}|D_{ijk}|^{2}\right)\nonumber\\
&&+\frac{n-2}{3(n-1)}\langle \iota_{\nabla f}W,C\rangle.
\end{eqnarray}
Hence, combining (\ref{nablaRIC}), (\ref{Ricnablaf}) and (\ref{auxintDR}), it is easy to check that
\begin{eqnarray}\label{auxintDR2}
\frac{2}{3}\int_{M}D_{ijk}R_{jk}\nabla_{i}f dM_{g}&=&\int_{M}(\mu+f)\left(\frac{1}{n-1}|\nabla Ric|^{2}-\frac{2n-1}{9(n-1)}|D_{ijk}|^{2}\right)dM_{g}\nonumber\\
&&-2\int_{M}(\mu+f)(R_{ij}R_{ik}R_{jk}-R_{ijkl}R_{ik}R_{jl})dM_{g}\nonumber\\
&&+\frac{n-2}{3(n-1)}\int_{M}\langle \iota_{\nabla f}W,C\rangle dM_{g}.
\end{eqnarray}

For our purpose, integrating by parts on the first term of the left-hand of (\ref{auxintDR2}), we immediately get
\begin{eqnarray}\label{auxintDR3}
\frac{1}{9(n-1)}\int_{M}(\mu+f)|D_{ijk}|^{2}dM_{g} &=&\frac{2}{3}\int_{M}(\mu+f)\langle {\rm div} D,Ric\rangle dM_{g}\nonumber\\
&&+\frac{1}{n-1}\int_{M}(\mu+f)|\nabla Ric|^{2}dM_{g}\nonumber\\
&&-2\int_{M}(\mu+f)(R_{ij}R_{ik}R_{jk}-R_{ijkl}R_{ik}R_{jl})dM_{g}\nonumber\\
&&+\frac{n-2}{3(n-1)}\int_{M}\langle \iota_{\nabla f}W,C\rangle dM_{g}.
\end{eqnarray}

However, using the fact that
$$R_{ij}R_{jk}R_{ik}-R_{ijkl}R_{ik}R_{jl}=\frac{1}{n-1}R|\mathring{Ric}|^{2}+\frac{n}{n-2}tr(\mathring{Ric}^{3})-W_{ijkl}R_{ik}R_{jl},$$
see \cite[Lemma 4]{balt18} for more details, we can take the integral in both side of the expression in Proposition~\ref{bochnerAUX} in order to get another  integral identity, that is,
\begin{eqnarray*}
0&=&\int_{M}(\mu+f)\Big(\frac{n-2}{n-1}|C_{ijk}|^{2}+|\nabla Ric|^{2}\Big)dM_{g}+ \frac{n\lambda}{n-1}\int_{M}|\mathring{Ric}|^{2}dM_{g}\nonumber\\
&&+2\int_{M}(\mu+f)(R_{ij}R_{ik}R_{jk}-R_{ijkl}R_{ik}R_{jl})dM_{g}-\frac{n-2}{n-1}\int_{M}\langle \iota_{\nabla f}W,C\rangle dM_{g}.
\end{eqnarray*}
Such formula can be rewritten using (\ref{normaD}) as follows
\begin{eqnarray}\label{INTDaux}
\frac{n-2}{3(n-1)}\int_{M}(\mu+f)|D_{ijk}|^{2}dM_{g}&=&\frac{4n-7}{n-1}\int_{M}(\mu+f)|\nabla Ric|^{2}dM_{g} +\frac{n\lambda}{n-1}\int_{M}|\mathring{Ric}|^{2}dM_{g}\nonumber\\
&&+2\int_{M}(\mu+f)(R_{ij}R_{ik}R_{jk}-R_{ijkl}R_{ik}R_{jl})dM_{g}\nonumber\\
&&-\frac{n-2}{n-1}\int_{M}\langle \iota_{\nabla f}W,C\rangle dM_{g}.
\end{eqnarray}
Therefore, it is enough to add the expressions obtained in (\ref{auxintDR3}) and (\ref{INTDaux}) to get the desired result.
\end{proof}

As an immediate consequence of the previous theorem we get the following result.

\begin{corollary}
Let $(M^n,\,g,\,f,\,\mu,\,\lambda),$ $n\geq3$,  be a critical space with $D_{ijk}=0.$  Suppose that
\begin{equation}\label{iWinq}
\theta\int_{M}\langle \iota_{\nabla f}W,C\rangle dM_{g}\leq0,
\end{equation}
where $\theta=1$ if $(\mu,\lambda)=(0,0)$ or $(0,1),$ and $\theta=-1$ if $(\mu,\lambda)=(1,-R/n).$ Then the following assertions holds:
\begin{itemize}
\item[(1)] $(M,g)$ has parallel Ricci tensor when $(\mu,\lambda)=(0,0)$;
\item[(2)] $(M,g)$ is an Einstein manifold when $(\mu,\lambda)=(0,1)$ or $(1,-R/n).$
\end{itemize}
In particular, for $n=3,$ $W_{ijkl}$ vanishes identically and the inequality (\ref{iWinq}) is trivially satisfied.
\end{corollary}

\begin{proof}
Since we are supposing the manifold satisfying the cyclic parallel Ricci tensor condition, it is immediate to apply the Theorem~\ref{thmMainC} to infer
\begin{eqnarray}\label{EqDnull}
0&=&\frac{3(2n-3)}{n-2}\int_{M} (\mu+f)|\nabla Ric|^{2}dM_{g}+\frac{3n\lambda}{2(n-2)}\int_{M} |\mathring{Ric}|^{2}dM_{g}\nonumber\\
&&-\int_{M}\langle \iota_{\nabla f}W,C\rangle dM_{g}.
\end{eqnarray}
Hence, if $\theta=1$ in (\ref{iWinq}), by (\ref{EqDnull}) we have that $M$ has parallel Ricci tensor. In addition, notice that for $(\mu,\lambda)=(0,1),$ we obtain necessarily that $M$ is an Einstein manifold.

In the sequel, considering $\theta=-1,$ which correspond to CPE metrics, we can rewrite the Equation (\ref{EqDnull}) as follows
\begin{eqnarray*}
\frac{3R}{2(n-2)}\int_{M} |\mathring{Ric}|^{2}dM_{g}+\int_{M}\langle \iota_{\nabla f}W,C\rangle dM_{g}=\frac{3(2n-3)}{n-2}\int_{M} (1+f)|\nabla Ric|^{2}dM_{g}.
\end{eqnarray*}
To conclude, it is sufficient to take the integral in both side of the formulae obtained from Lemma~\ref{L4}, in order to arrive at
$$\int_{M} (1+f)|\nabla Ric|^{2}dM_{g}=0.$$
Therefore, returning to the previous integral identity, we get again that $M$ is an Einstein Manifold, as desired.
\end{proof}

Now, our classification results for a $3$-dimensional critical spaces with cyclic parallel Ricci tensor follows directly by previous corollary jointly with the well-known classification of such spaces under parallel Ricci tensor condition, see Theorems \ref{CHYCPE}, \ref{KL} and \ref{thmMT} in the introduction.

\subsection{Miao-Tam critical metrics with $R\leq0$}

In this last section, we focus on the classification of an $n$-dimensional Miao-Tam critical metric with non-positive scalar curvature and satisfying the cyclic parallel Ricci tensor condition. To do this, we will apply the Lemma~\ref{L4} in order to get a new integral identity, which in this case, we will need to provide a specific control over the boundary term. More precisely, we established the following result.

\begin{theorem}\label{thmMTintID}
Let $(M^n,\,g,\,f),$ be a Miao-Tam critical metric. Then we have:
\begin{equation*}
\frac{1}{3}\int_{M} f|D_{ijk}|^{2}dM_{g}=2\int_{M} f|\nabla Ric|^{2}dM_{g}-\frac{R}{n-1}\int_{M} f|\mathring{Ric}|^{2}dM_{g}+\int_{\partial M}|\mathring{Ric}|^{2}|\nabla f|dS.
\end{equation*}
\end{theorem}

\begin{proof} First of all,  considering the vector field $X$ whose the component functions are given by
$$X_{i}=R_{ik}R_{kj}\nabla_{j}f +R_{ijkl}\nabla_{l}f R_{jk},$$
we may apply the divergence theorem to infer
\begin{eqnarray}\label{auxDIVX}
\int_{M} {\rm div} XdM_{g}&=&\int_{\partial M} -R_{ik}R_{kj}\nabla_{j}f\frac{\nabla_{i} f}{|\nabla f|} - R_{ijkl}\nabla_{l}f R_{jk}\frac{\nabla_{i} f}{|\nabla f|}dS\nonumber\\
&=&\int_{\partial M}\frac{1}{|\nabla f|}[-|Ric(\nabla f)|^{2}- R_{ijkl}\nabla_{l}f R_{jk}\nabla_{i} f]dS.
 \end{eqnarray}
Taking into account that $f$ vanishing on $\partial M$, we substitute Lemma~\ref{L1} in (\ref{auxDIVX}) to arrive at
\begin{eqnarray*}
\int_{M} {\rm div} XdM_{g}&=&\int_{\partial M}\frac{1}{|\nabla f|}\Big[-|Ric(\nabla f)|^{2}+\frac{R}{n-1}(\nabla_{i}fg_{jk}-\nabla_{j}fg_{ik})R_{jk}\nabla_{i} f\\
&&-(\nabla_{i}fR_{jk}-\nabla_{j}fR_{ik})R_{jk}\nabla_{i} f\Big]dS\\
&=&\int_{\partial M}\frac{1}{|\nabla f|}\Big[ -|Ric|^{2}|\nabla f|^{2}+\frac{R^{2}}{n-1}|\nabla f|^{2}-\frac{R}{n-1}Ric(\nabla f,\nabla f)\Big]dS,
 \end{eqnarray*}
which can be rewritten as
\begin{eqnarray*}
\int_{M} {\rm div} XdM_{g}&=&-\int_{\partial M}|\mathring{Ric}|^{2}|\nabla f|dS -\frac{R}{n-1}\int_{\partial M}\frac{1}{|\nabla f|}\mathring{Ric}(\nabla f,\nabla f)dS.
 \end{eqnarray*}
Next, using Lemma~\ref{AuxRicci} it is easy to check that
\begin{eqnarray*}
\int_{M} {\rm div} XdM_{g}&=&-\int_{\partial M}|\mathring{Ric}|^{2}|\nabla f|dS +\frac{R}{n-1}\int_{M}f|\mathring{Ric}|^{2}dS,
 \end{eqnarray*}
i.e.,
\begin{eqnarray}\label{eq:divX}
\int_{M}\nabla_{i}(R_{ik}R_{kj}\nabla_{j}f+R_{ijkl}\nabla_{l}f R_{jk})dM_{g}&=&-\int_{\partial M}|\mathring{Ric}|^{2}|\nabla f|dS\nonumber\\ &&+\frac{R}{n-1}\int_{M}f|\mathring{Ric}|^{2}dS.
\end{eqnarray}

Finally, the Equation (\ref{eq:divX}) replaced in Lemma~\ref{L4} provides the requested result.
\end{proof}

\subsection{Conclusion of the proof of Theorem \ref{thmMTnD}}
\begin{proof}

Taking into account that $M$ has cyclic parallel Ricci tensor, it suffices to invoke Theorem \ref{thmMTintID} to obtain
$$2\int_{M} f|\nabla Ric|^{2}dM_{g}-\frac{R}{n-1}\int_{M} f|\mathring{Ric}|^{2}dM_{g}+\int_{\partial M}|\mathring{Ric}|^{2}|\nabla f|dS=0.$$
Therefore, as we are considering the non-positive scalar curvature case, it is immediate to verify using the above integral identity, that $M^n$ has parallel Ricci tensor. Then, we are in position to use Theorem \ref{thmMT} (see also Theorem 1.1 of  \cite{miaotamTAMS}) to conclude that $(M^n ,\,g)$ is isometric to a geodesic ball in a simply connected space form $\Bbb{R}^{n}$ or $\Bbb{H}^{n}.$  So, the proof is completed.
\end{proof}

\begin{acknowledgement}
The authors would like to thank Rondinelle Batista, Wilson Cunha, Manoel Vieira and Kelton Bezerra for helpful discussions about this subject.
\end{acknowledgement}

\end{document}